\documentclass[11pt,twoside,a4paper,reqno]{amsart}

\usepackage{amsmath}
\usepackage{amsthm}
\usepackage{amsfonts, amssymb}
\usepackage[all]{xy}
\usepackage{url}
\usepackage{wrapfig}

\usepackage{latexsym}

\usepackage{graphicx}
\usepackage{graphics}
\usepackage[scriptsize,bf]{caption}
\usepackage{float}
\usepackage{enumerate}
\usepackage{verbatim}

\theoremstyle{plain}
\newtheorem{lema}{Lemma}[section]
\newtheorem{prop}[lema]{Proposition}
\newtheorem{teo}[lema]{Theorem}

\newtheorem{coro}[lema]{Corollary}
\newtheorem*{teoA}{Theorem A}
\newtheorem*{teoB}{Theorem B}
\theoremstyle{definition}
\newtheorem*{defi}{Definition}
\newtheorem{ej}[lema]{Example}

\newtheorem{obs}[lema]{Remark}
\newtheorem{obss}[lema]{Remarks}

\newcommand{\M}{\mathfrak{M}}
\def\set#1{\{#1\}}
\def\r{\mathcal{R}}
\begin{document}

\title[A complex is determined by its set of discrete Morse functions]{A simplicial complex is uniquely determined by its set of discrete Morse functions}

\author[N. A. Capitelli and E. G. Minian]{Nicolas Ariel Capitelli and Elias Gabriel Minian}

\thanks{Researchers of CONICET. Partially supported by grants ANPCyT PICT-2011-0812, CONICET PIP 112-201101-00746 and UBACyT 20020130100369.}



\address{Departamento de Matem\'atica-IMAS\\
 FCEyN, Universidad de Buenos Aires\\ Buenos
Aires, Argentina.}

\email{ncapitel@dm.uba.ar} \email{gminian@dm.uba.ar}

\begin{abstract} We prove that a connected simplicial complex is uniquely determined by its complex of discrete Morse functions. This settles a question raised by Chari and Joswig. In the $1$-dimensional case, this implies that the complex of rooted forests of a connected graph $G$ completely determines $G$. 
\end{abstract}

%The complex of discrete Morse functions of a simplicial complex $K$ was introduced by Chari and Joswig in 2005 in order to extract information of $K$ from its deformation properties. 

\maketitle

\section{Introduction}

The \emph{complex of discrete Morse functions} $\M(K)$ of a finite simplicial complex $K$ was introduced by Chari and Joswig in \cite{ChJo} to study the topology of simplicial complexes in terms of their sets of discrete deformations. Despite the potential utility of this complex, very little was known about the relationship between $K$ and $\M(K)$. Chari and Joswig studied some properties of the complexes associated to graphs and simplices and computed the homotopy type of the complex associated to the $2$-simplex. Their work was shortly followed  by Ayala, Fern\'andez, Quintero and Vilches, who described the structure of the \emph{pure Morse complex} of a graph $G$, i.e. the subcomplex of $\M(G)$ generated by the simplices of maximal dimension  \cite{AyFeQuVi1}. As pointed out in \cite{ChJo}, the construction of $\M(K)$ in the context of graphs was already implicit in the work of Kozlov \cite{Koz1}, who studied complexes arising from directed sub-trees of a given (directed) graph. Kozlov proved shellability of the complexes associated to complete graphs and computed the homotopy type of the complexes associated to paths and cycles.

The aim of this article is to settle the connection between a simplicial complex and its complex of discrete Morse functions. We show that $K$ is completely determined by $\M(K)$. Concretely, our main result is the following.

\begin{teoA}\label{Theorem:MainForComplexes} Let $K,L$ be finite connected simplicial complexes. If $\M(K)$ is isomorphic to $\M(L)$ then $K$ is isomorphic to $L$.\end{teoA}

For the $1$-dimensional case, we prove that Theorem A also holds for multigraphs.

\begin{teoB}\label{Theorem:MainForMultigraphs} Let $G,G'$ be finite connected multigraphs. If $\M(G)$ is isomorphic to $\M(G')$ then $G$ is isomorphic to $G'$.\end{teoB}

We also exhibit an example which shows that the homotopy type of $\M(K)$ does not determine the homotopy type of $K$. 

The results in this article provide the complete answers to the foundational questions about $\M(K)$ raised by Chari and Joswig in \cite{ChJo}.

\section{The complex of discrete Morse functions}

All simplicial complexes that we deal with are assumed to be finite. We write $\sigma\prec\tau$ if the simplex $\sigma$ is an immediate face of $\tau$ (i.e. a proper maximal face) and we let $V_K$ denote the set of vertices of a complex $K$. We denote by $\Delta^n$ the standard complex consisting of all the faces of an $n$-simplex, and by $\partial\Delta^n$ its boundary (i.e. the complex of all the proper faces of the simplex).

A \emph{discrete Morse funcion} $f$ on an abstract simplicial complex $K$ is a map $f:K\rightarrow \mathbb{R}$ satisfying, for every $\sigma\in K$,\begin{enumerate}
\item\label{item:Property1DefinitionMorseFunction} $|\{\tau\succ\sigma\,|\,f(\tau)\leq f(\sigma)\}|\leq 1$ and
\item\label{item:Property2DefinitionMorseFunction} $|\{\nu\prec\sigma\,|\,f(\nu)\geq f(\sigma)\}|\leq 1$.
\end{enumerate}
Here $|A|$ denotes the cardinality of the set $A$. A simplex $\sigma$ such that both of these numbers are zero is called \emph{critical}. If $f(\eta)\geq f(\rho)$ for some $\eta\prec\rho$ then the pair $(\eta,\rho)$ is called a \emph{regular pair}. One can easily see that every simplex in $K$ is either critical or belongs to a unique regular pair (see \cite{For1, For2} for more details). If $(\sigma,\tau)$ is a regular pair, we call $\sigma$ the \emph{source simplex} of the pair, and write $s(\sigma,\tau)=\sigma$, and we call $\tau$ the \emph{target simplex} of the pair, and  write $t(\sigma,\tau)=\tau$. Typically, a regular pair $(\sigma,\tau)$ is depicted graphically as an arrow from $\sigma$ to $\tau$ (see Figure \ref{Figure:ExamplesOfArrowsInFMD}).

\begin{figure}[h]
\centering
\includegraphics[scale=0.7]{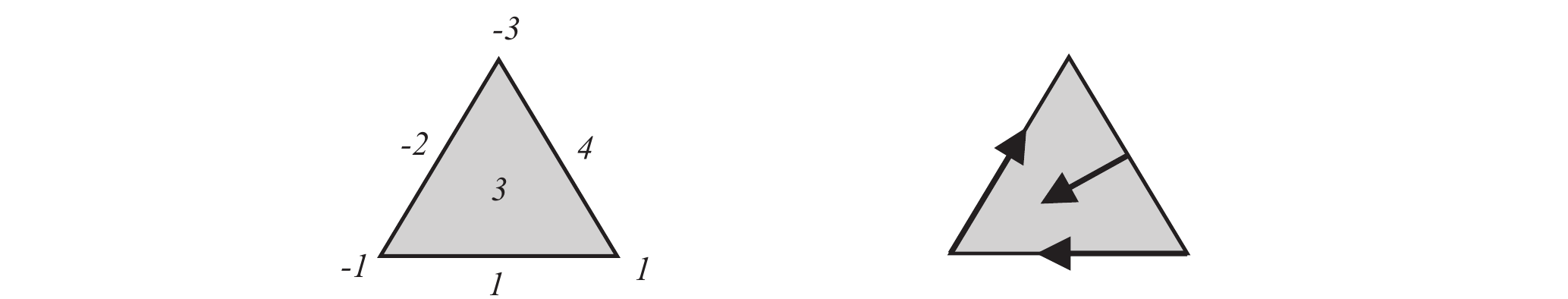}
\caption{Graphical representation of regular pairs.}
\label{Figure:ExamplesOfArrowsInFMD}
\end{figure}

The index of a regular pair $(\sigma,\tau)$ is the dimension of $\sigma$. A regular pair of index $k$ will be sometimes denoted by $(\sigma^k,\tau^{k+1})$. Given two discrete Morse functions $f,g$ on $K$ we write $f\lesssim g$ if every regular pair of $f$ is also a regular pair of $g$. Following \cite{ChJo}, if $f\lesssim g$ and $g\lesssim f$ (i.e. both functions have the same regular pairs) then we say that they are \emph{equivalent}.
We will make no distinction between equivalent Morse functions, i.e. we will work with classes of discrete Morse functions under this equivalence relation.

A discrete Morse function with exactly one regular pair is called a \emph{primitive Morse function}. We will often identify a primitive Morse function with its sole regular pair. A collection $f_0,\ldots,f_r$ of primitive Morse functions is said to be \emph{compatible} if there exists a discrete Morse function $f$ on $K$ with $f_i\lesssim f$ for every $i=0,\ldots,r$. The \emph{complex of discrete Morse functions} of $K$ is the simplicial complex $\M(K)$ whose vertices are the primitive Morse functions on $K$ and whose $r$-simplices are the discrete Morse functions with $r+1$ regular pairs. We identify in this way a discrete Morse function $f$ with the set $\{f_0,\ldots,f_r\}$ of all primitive Morse functions satisfying $f_i\lesssim f$ (i.e. the set of its regular pairs). $\M(K)$ is also called the \emph{discrete Morse complex} of $K$. Figure \ref{Figure:ExamplesOfComplexOfFMD} shows some low-dimensional examples of discrete Morse complexes.

%sets $\Lambda=\{f_1,\ldots,f_r\}$ of compatible primitive Morse functions.

\begin{figure}[h]
\centering
\includegraphics[scale=0.6]{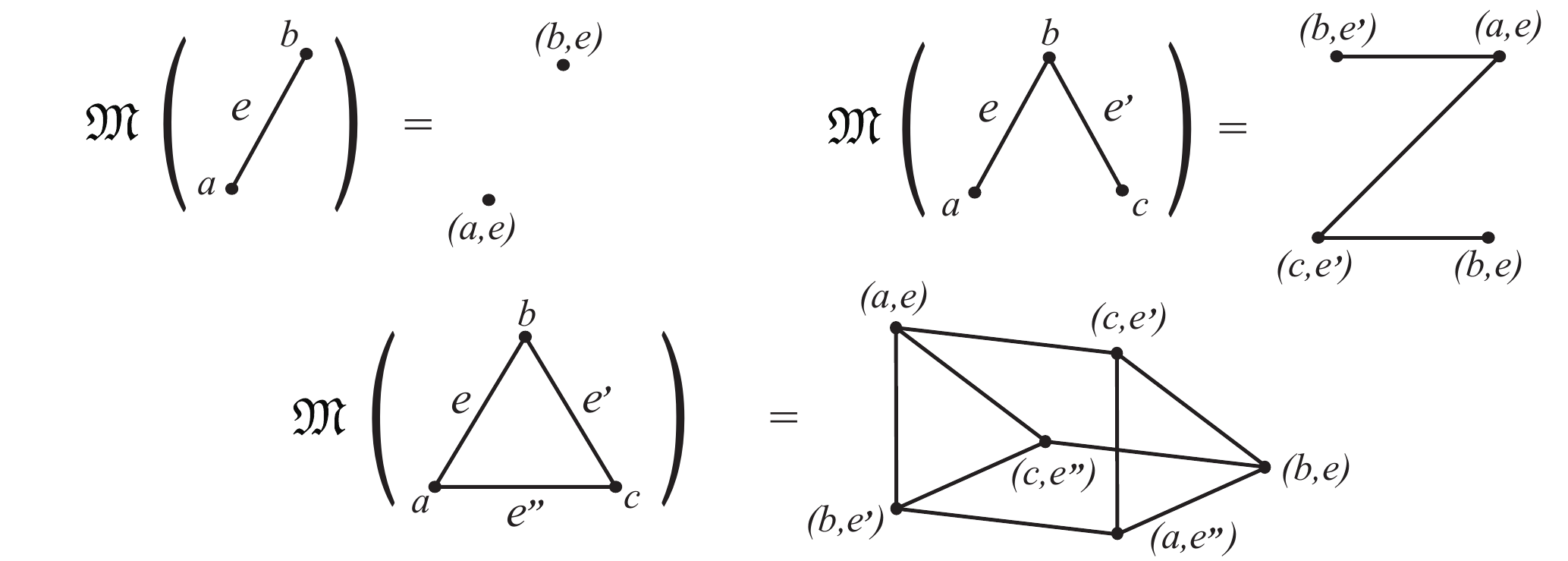}
\caption{Examples of complexes of discrete Morse functions.}
\label{Figure:ExamplesOfComplexOfFMD}
\end{figure}

There is an alternative approach to discrete Morse theory due to Chari \cite{Cha} where the deformations are encoded in terms of acyclic matchings in the Hasse diagram of the face poset of the simplicial complex. It is not hard to see that the pairing of simplices which form regular pairs of a discrete Morse function determines a matching in the Hasse diagram $\mathcal{H}_K$ of $K$. If the arrows in this matching are reversed, it can be easily shown that the resulting directed graph is acyclic. On the other hand, from an acyclic matching on the Hasse diagram of a simplicial complex one can build a discrete Morse function $f$ on $K$ where the regular pairs of $f$ are precisely the edges of the matching. From this viewpoint, $\M(K)$ is the simplicial complex on the edges of the Hasse diagram of $K$ whose simplices are the subsets of edges which form acyclic matchings.

%This is a natural perspective since the relevant information of a Morse function lies in its regular pairs.
%\begin{figure}[h]
%\centering
%\includegraphics[scale=0.75]{hasse}
%\caption{Regular pairs in terms of acyclic matchings on $\mathcal{H}_K$: the arrows in the matching (in bold) have been reversed.}
%\label{Figure:hasse}
%\end{figure}

\section{The complexes associated to graphs}

The complex of discrete Morse functions has been studied almost exclusively for graphs, as the construction of $\M(K)$ for a general $K$ is rather complicated (see for example \cite{AyFeQuVi1,ChJo}). We focus first on this case and settle the main result for $1$-dimensional regular CW-complexes (Theorem B).

Recall that a \emph{multigraph} $G$ is a triple $(V_G,E_G,f_G)$ where $V_G$ is a (finite) set of vertices, $E_G$ is a set of edges and $f_G:E_G\rightarrow\{\{u,v\}\,:\,u,v\in V_G\text{ and }u\neq v\}$ is a map which assigns to each edge its boundary vertices. If $f_G(e)=f_G(e')$ for $e,e'\in E_G$, we say that $e,e'$ are \emph{parallel edges}. For $v,v'\in V_G$, $E_G(v,v')$ will stand for the set of parallel edges between $v$ and $v'$. Note that, by definition, a multigraph has no loops. \emph{Simple} graphs correspond to multigraphs $G$ where $f_G$ is injective. In this case we shall identify an edge with its boundary vertices and write $e=vw$ if $f_G(e)=\{v,w\}$. Note that simple graphs are precisely the $1$-dimensional simplicial complexes and multigraphs are precisely the $1$-dimensional regular CW-complexes (see \cite{LuWe} for the necessary definitions).

The complex of discrete Morse functions of a graph was first studied by Kozlov \cite{Koz1} under a different context. Given a directed graph $G$, Kozlov defined the simplicial complex $\Delta(G)$ whose vertices are the edges of $G$ and whose faces are all directed forests which are subgraphs of $G$. In \cite{Koz1} he studied the shellability of the complete double-directed graph on $n$ vertices (a graph having exactly one edge in each direction between any pair of vertices) and computed the homotopy type of the double-directed $n$-cycle and the double-directed $n$-path. It is not hard to see that for any (undirected) graph $G$, the identity $\M(G)=\Delta(d(G))$ holds, where $d(G)$ is the directed graph on the vertices of $G$ with one edge in each direction between adjacent vertices of $G$. The aforementioned examples studied by Kozlov correspond respectively to the complex of Morse functions of the complete graph, the $n$-cycle and the $n$-path. Complexes of directed graphs have been widely studied (see for example \cite{BjWe,Eng,Joj,Koz1}) and some results of this theory were used in Babson and Kozlov's proof of the Lov\'asz conjecture (see \cite{BaKo}).

%It is worth mentioning that Kozlov \cite{Koz1} also provides a recursive procedure to compute the homology of $\Delta(G)$ when $G$ is (essentially) a tree.

In this section we prove Theorem B, which is the special case of Theorem A for regular $1$-dimensional CW-complexes. The definition of the complex of Morse functions for regular CW-complexes is identical to the simplicial case. In particular, for a multigraph $G$, $\M(G)$ can be viewed as the simplicial complex with one vertex for each directed edge in $G$ and whose simplices are the collections of directed edges which do not form directed cycles.

We first establish the result for simple graphs (i.e. the $1$-dimensional case of Theorem A) and then extend it to general multigraphs. We begin by collecting some basic facts about the discrete Morse complex of simple graphs. 

Given two simplicial complexes $K, L$, we denote $K\equiv L$ if they are isomorphic. 

\begin{lema} Let $G$ be a connected simple graph. Then,\begin{enumerate}
\item $|V_{\M(G)}|=2|E_G|$.
\item $\dim(\M(G))=|V_G|-2$.
\end{enumerate}\end{lema}

\begin{proof} If $G$ is a tree then it is collapsible and there exists a discrete Morse function $f\in\M(G)$ for which all the edges of $G$ are regular (see \cite[Lemma 4.3]{For1}). Hence, $\dim(\M(G))=|E_G|-1=|V_G|-2$. For the general case, proceed by induction on $n=|E_G|$. If $G$ is not a tree, let $f\in\M(G)$ be of maximal dimension and let $e_0,\ldots,e_r$ be a cycle in $G$. There must be an edge $e_i$ which is not regular for $f$ (see \cite[Theorem 9.3]{For1}). Let $G'=G-\{e_i\}$. $G'$ is still connected because $e_i$ is in a cycle, $|E_{G'}|=|E_G|-1$ and, by induction, $\dim(\M(G'))=|V_{G'}|-2=|V_G|-2$. Since $f\in\M(G')$ and $\dim(\M(G'))\leq\dim(\M(G))$, then $\dim(\M(G))=|V_G|-2$.\end{proof}

\begin{coro} If $G,G'$ are connected simple graphs such that $\M(G)\equiv\M(G')$ then $|V_G|=|V_{G'}|$ and $|E_G|=|E_{G'}|$. In particular their fundamental groups $\pi_1(G)$ and $\pi_1(G')$ are isomorphic.\end{coro}

\begin{obs}\label{Obs:compatibility} It is easy to check that a vertex $v\in G$ is a leaf if and only if the vertex $(v,e)\in V_{\M(G)}$ is compatible with every other $(u,e')\in V_{\M(G)}$ with the unique exception of $(w,e)$, where $w$ is the other vertex of the edge $e$. This happens if and only if $\deg(v,e)=2|E_G|-2$, where $\deg(v,e)$ is the degree of the vertex $(v,e)$ in the $1$-skeleton $\M(G)^{(1)}$ (i.e. the subcomplex of $\M(G)$ consisting of the simplices of dimension $\leq 1$). In particular, if $\M(G)\equiv\M(G')$ then $G$ and $G'$ have the same number of leaves.\end{obs}

Let $C_n$ denote the simple cycle with $n$ vertices.

\begin{coro}\label{Corollary:MainForCycles} Let $G,G'$ be two connected simple graphs. If $\M(G)\equiv\M(G')$ and $G=C_n$ then $G'=C_n$.\end{coro}

\begin{proof} By a previous result, $|V_G|=|V_{G'}|$ and $|E_G|=|E_{G'}|$. Since $G=C_n$ then $|V_G|=|E_G|$ and therefore $|V_{G'}|=|E_{G'}|$. Also, since $G$ has no leaves then $G'$ has no leaves. Therefore, $G'=C_n$.\end{proof}

In order to prove the main results of this paper we will analyze compatibility of regular pairs, similarly as we did in Remark \ref{Obs:compatibility}. From now on, we write $(\sigma,\tau)\sim (\eta,\rho)$ if $(\sigma,\tau)$ and $(\eta,\rho)$ are compatible as primitive Morse functions (i.e. if they form a simplex in $\M(K)$), and $(\sigma,\tau)\nsim (\eta,\rho)$ whenever they are not.

\begin{teo}\label{Theorem:MainForSimpleGraphs} Let $G,G'\neq C_n$ be connected simple graphs and let $F:\M(G)\rightarrow\M(G')$ be a simplicial isomorphism. Define a mapping $f:G\rightarrow G'$ by $f(v)=s(F(v,e))$, where $e$ is any edge incident to $v$. Then $f$ is a well-defined simplicial isomorphism.\end{teo}

\begin{proof} The key part of the proof is to see that $f$ is well-defined, i.e. that $f(v)$ does not depend on the choice of the incident edge $e$. Suppose otherwise and let $(v,e_0),(v,e_1)\in V_{\M(K)}$ be such that $F(v,e_0)=(w,a)$ and $F(v,e_1)=(w',b)$ with $w\neq w'$. Since $(v,e_0)\nsim (v,e_1)$ then $(w,a)\nsim(w',b)$ and hence $a=b$ (see Figure \ref{Figure:InitialSituation}).

\begin{figure}[h]
\centering
\includegraphics[scale=0.7]{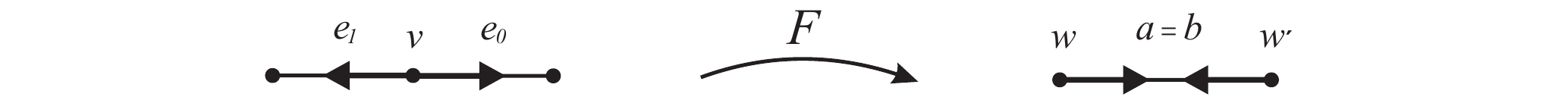}
\caption{}
\label{Figure:InitialSituation}
\end{figure}

We claim that under this situation we can choose such a vertex $v$ of $G$ with degree greater than or equal to $3$. This will lead to a contradiction since an edge containing $v$ different from $e_0$ and $e_1$ provides a primitive Morse function on $G$ which is incompatible with both $(v,e_0)$ and $(v,e_1)$, while the simplicity of $G'$ implies that there is no possible primitive Morse function on $G'$ incompatible with both $(w,a)$ and $(w',a)$. To prove this claim, let $e_1=vv'$ and consider the primitive Morse function $(v',e_1)$. Since $(w',a)=F(v,e_1)\nsim F(v',e_1)$ and $F$ is an isomorphism then there exists and edge $c=w'w''\in G'$  such that $F(v',e_1)=(w',c)$. Consider now $(w'',c)\in\M(G')$. Using a similar argument for $F^{-1}$ and $(w'',c)$ one can find  an edge $e_2\neq e_0,e_1$ such that $F^{-1}(w'',c)=(v',e_2)$ (see Figure \ref{Figure:InductiveSituation}).

\begin{figure}[h]
\centering
\includegraphics[scale=0.7]{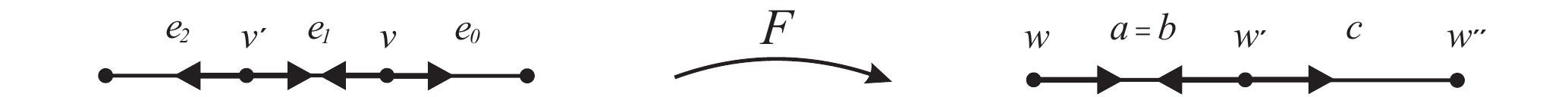}
\caption{}
\label{Figure:InductiveSituation}
\end{figure}

Note that the primitive Morse functions $(v',e_1), (v',e_2)$ satisfy the same hypotheses than $(v,e_0),(v,e_1)$ (but replacing $(w,a),(w',a)$ with $(w',c),(w'',c)$ respectively). Repeating this argument we obtain a path $e_1,e_2,e_3,\ldots$ where, for any vertex $v\in e_i\cap e_{i+1}$, $(v,e_i),(v,e_{i+1})$ are mapped to primitive Morse functions on $G'$ of the form $(u,d),(u',d)$ with $u\neq u'$. By finiteness, this path must form a cycle $C=\{e_j,e_{j+1},\ldots,e_{j+k-1},e_{j+k}=e_j\}$ for some $j,k$. If $j=0$, and since $G$ is not a cycle, there is by connectedness an edge $e\notin C$ intersecting $C$. In this case, $x= e\cap C$ is the desired vertex (see Figure \ref{Figure:CycleSituation} ($a$)). If $j>0$ then the vertex $y=e_{j-1}\cap e_j$ is the desired vertex (see Figure \ref{Figure:CycleSituation} ($b$)). This proves that $f$ is well-defined.
\begin{figure}[h]
\centering
\includegraphics[scale=0.7]{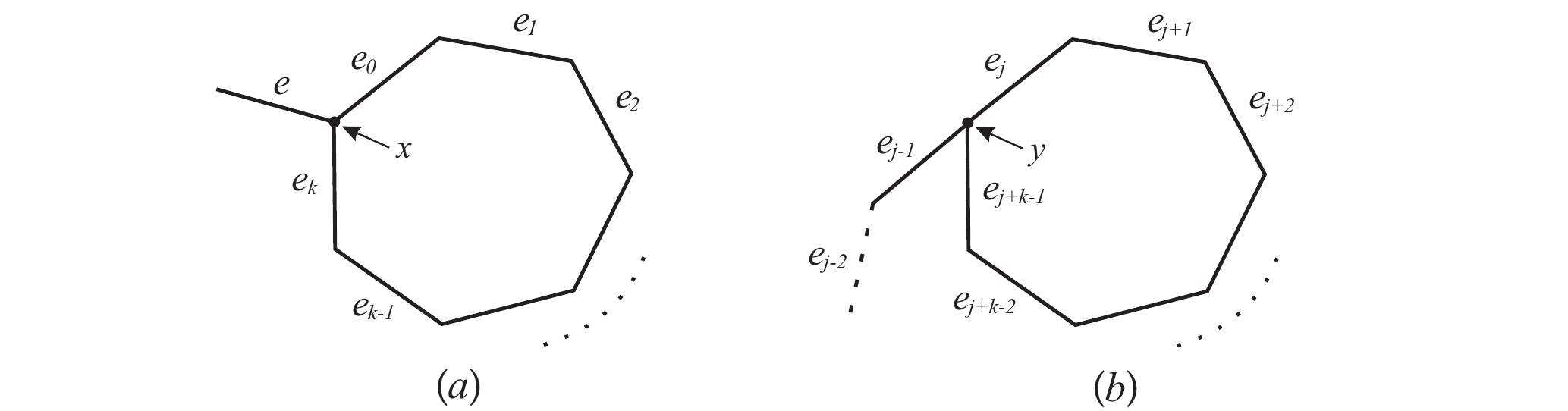}
\caption{}
\label{Figure:CycleSituation}
\end{figure}

We show now that $f$ is a simplicial morphism. Consider an edge $e=vv'\in G$. We must see that $f(v)f(v')\in G'$. Since $(v,e)\nsim(v',e)$ then $F(v,e)\nsim F(v',e)$. Therefore, either $s(F(v,e))=s(F(v',e))$ or $t(F(v,e))=t(F(v',e))$. In the first case, the same reasoning as above applied to $h=s\circ F^{-1}:G'\rightarrow G$ gives a contradiction. Therefore, $t(F(v,e))=t(F(v',e))$ and, in particular, $f(v)f(v')\in t(F(v,e))$ is an edge in $G'$.

Finally, it is easy to see that $f^{-1}=s\circ F^{-1}$ is the inverse of $f$.\end{proof}

\begin{coro}\label{Corollary:MainTheoremForSimpleGraphs} Let $G,G'$ be connected simple graphs. If $\M(G)\equiv\M(G')$ then $G\equiv G'$.\end{coro}

\begin{proof} Follows from Corollary \ref{Corollary:MainForCycles} and Theorem \ref{Theorem:MainForSimpleGraphs}.\end{proof}

We now extend the result to multigraphs. Two primitive Morse functions $(v,e),(v',e')\in\M(G)$ are said to be \emph{parallel} if $v=v'$ and $e$ is parallel to $e'$ in $G$. Recall that the \emph{link} of a simplex $\sigma\in K$ is the subcomplex $lk(\sigma,K)=\{\tau\in K:\ \tau\cap\sigma=\emptyset,\ \tau\cup\sigma\in K\}$. 

\begin{lema}\label{Lemma:ParallelCharacterization} Let $G$ be a connected multigraph with more than two vertices. Then two primitive Morse functions $(v,e),(v',e')$ are parallel in $\M(G)$ if and only if $(v,e)\nsim (v',e')$ and $lk((v,e),\M(G))=lk((v',e'),\M(G))$.\end{lema}

\begin{proof} Suppose first that $(v,e)\nsim (v',e')$ and $lk((v,e),\M(G))=lk((v',e'),\M(G))$. If $(v,e)$ and $(v',e')$ are not parallel in $\M(G)$, then there are only three possibilities for the edges $e$ and $e'$ in $G$ which are shown in Figure \ref{Figure:PossibilitiesArise}.

\begin{figure}[H]
\centering
\includegraphics[scale=0.7]{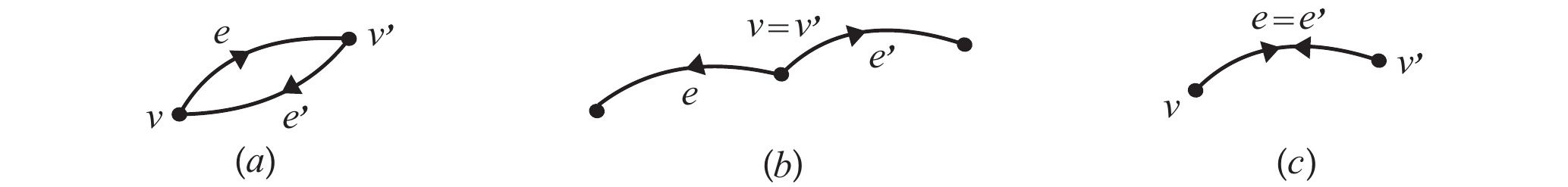}
\caption{}
\label{Figure:PossibilitiesArise}
\end{figure}

Since $|V_G|\geq 3$ and $G$ is connected, in each of the three cases, $G$ locally looks as in Figure \ref{Figure:PossibilitiesContradicted}.

\begin{figure}[H]
\centering
\includegraphics[scale=0.7]{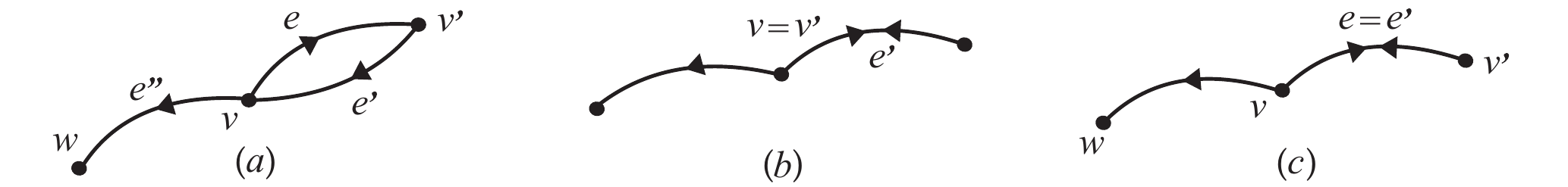}
\caption{}
\label{Figure:PossibilitiesContradicted}
\end{figure}

This contradicts the fact that  $lk((v,e),\M(G))=lk((v',e'),\M(G))$. The other implication is trivial.\end{proof}

Given a simplicial complex $K$, we define an equivalence relation $\r$ on $V_K$ as follows:
 $$v\r w\Leftrightarrow v=w \text{ or } \{v,w\}\notin K\text{ and }lk(v,K)=lk(w,K).$$
Let $\widetilde{K}$ be the simplicial complex whose vertices are the equivalence classes of vertices of $K$ and whose simplices are the sets $\{\tilde{v}_0,\ldots,\tilde{v}_r\}$ such that $\{v_0,\ldots,v_r\}\in K$. Here $\tilde v$ denotes the equivalence class of the vertex $v$. Note that $\widetilde{K}$ is well-defined since, if $v_i\r v_i'$ then $\{v_0,\ldots,v_i,\ldots,v_r\}\in K$ if and only if $\{v_0,\ldots,v_i',\ldots,v_r\}\in K$.

\begin{prop}\label{Proposition:QuotientIsomorphism} Let $K, L$ be simplicial complexes and let $\tilde K$ and $\tilde L$ be as above. If $f:K\rightarrow L$ is a simplicial isomorphism then the map $\tilde{f}:\widetilde{K}\rightarrow\widetilde{L}$ given by $\tilde{f}(\tilde{v})=\widetilde{f(v)}$ is a simplicial isomorphism.\end{prop}

\begin{proof} We prove first that $\tilde f$ is well-defined. Suppose $v\r v'$ with $v\neq v'$. Since $\{v,v'\}\notin K$ and $f$ is an isomorphism then $\{f(v),f(v')\}\notin L$. Also, if $\{f(v)\}\cup \sigma\in L$ then $\{v\}\cup f^{-1}(\sigma)\in K$, which implies that $\{v'\}\cup f^{-1}(\sigma)\in K$. Therefore $f(v')\cup\sigma \in L$.

Finally, $\tilde{f}$ is an isomorphism since $\tilde{f}^{-1}=\widetilde{f^{-1}}$.\end{proof}

\begin{defi} For a multigraph $G$ we define the \emph{simplification} of $G$, denoted by $sG$, as the simple graph obtained from $G$ by identifying parallel edges.\end{defi}

\begin{obs}\label{Remark:SimplificationGraphs}
By Lemma \ref{Lemma:ParallelCharacterization} one can check that the map $f:\widetilde{\M(G)}\to \M(sG)$ defined by $f(\widetilde{(v,e)})=(v,\overline{e})$ is a well-defined isomorphism. Here $\overline{e}$ is the image of the edge $e$ in $sG$.
\end{obs}

\begin{proof}[Proof of Theorem B] Let $F:\M(G)\rightarrow\M(G')$ be an isomorphism. By Proposition \ref{Proposition:QuotientIsomorphism} and Remark \ref{Remark:SimplificationGraphs}, $F$ induces an isomorphism $\M(sG)\rightarrow\M(sG')$ which we also denote by $F$. By Theorem \ref{Theorem:MainForSimpleGraphs} there is an isomorphism $f:sG\rightarrow sG'$ sending a vertex $v$ to $s(F(v,e))$ for any edge $e$ incident to $v$. Then, in order to see that $G$ and $G'$ are isomorphic, we only need to check that $|E_G(v,w)|=|E_{G'}(f(v),f(w))|$ for any pair of vertices $v,w$ of $G$.

We can suppose that $|E_G(v,w)|\neq 0$ and choose some $e\in E_G(v,w)$. Then $(v,e)\in\M(G)$ and let $e'=t(F(v,e))\in E_{G'}(f(v),f(w))$. Note that the set $E_G(v,w)$ is in bijection with the set $\{(v,a)\in\M(G),\ (v,a)\nsim (w,e)\}$. Similarly, $E_{G'}(f(v),f(w))$ is in bijection with $\{(f(v),a')\in\M(G'),\ (f(v),a')\nsim (f(w),e')\}$. By the isomorphism $F$, both sets have the same cardinality.
\end{proof}

Chari and Joswig asked in \cite{ChJo} whether there is any connection between the homotopy types of $K$ and $\M(K)$. They implicitly showed that the homotopy type of $K$ does not determine the homotopy type of $\M(K)$. For instance, by \cite[Proposition 5.1]{ChJo} the complex of Morse functions associated to the 1-simplex is homotopy equivalent to $S^0$ and the one associated to the 2-simplex is homotopy equivalent to $S^1\vee S^1\vee S^1\vee S^1$. The following example shows that the homotopy type of $\M(K)$ does not determine the homotopy type of $K$ either.

\begin{ej} Consider the following simple graphs. $G$ has three vertices $u,v,w$ and two edges $uv,uw$. The graph $G'$ has four vertices $a,b,c,d$ and four edges $ab,bc,ac,ad$. Note that they are not homotopy equivalent while their associated complexes of Morse functions are both contractible.\end{ej}

\section{Proof of the main result}

We now extend the result of Corollary \ref{Corollary:MainTheoremForSimpleGraphs} to simplicial complexes of any dimension. The idea behind the proof is that, in ``almost all" cases, a simplicial isomorphism $F:\M(K)\rightarrow\M(L)$ restricts to an isomorphism $F|_{\M(K^{(1)})}:\M(K^{(1)})\rightarrow\M(L^{(1)})$ between the $1$-skeletons and by Theorem \ref{Theorem:MainForSimpleGraphs} the $1$-skeletons of $K$ and $L$ are isomorphic. Then an inductive argument shows that an isomorphism $\M(K)\equiv\M(L)$ forces all skeletons of $K$ and $L$ to be isomorphic.

In the following we will use Forman's concept of $V$-path associated to a discrete vector field $V$ over a complex $K$. Given a discrete Morse function $f:K\rightarrow\mathbb{R}$, an \emph{$f$-path of index $k$} is a sequence of regular $k$-simplices $\sigma_0,\ldots,\sigma_r\in K$ such that $\sigma_i\neq\sigma_{i+1}$ for all $0\leq i\leq r-1$ and $\sigma_{i+1}\prec \tau_i$, where $\tau_i$ is the target of the regular pair with source $\sigma_i$. This is actually the notion of a \emph{$V_f$-path}, where $V_f$ is the discrete gradient vector field of $f$. The $f$-path is called \emph{closed} if $\sigma_0=\sigma_r$ and \emph{non-stationary} if $\sigma_0\neq\sigma_1$. We shall be exclusively dealing with non-stationary closed $f$-paths, so we will simply refer to them as  \emph{$f$-cycles}. Note that an $f$-cycle of index $k$ is equivalent to having an incompatible collection $P=\{(\sigma_0,\tau_0),\ldots,(\sigma_r,\tau_r)\}$ of primitive Morse functions of index $k\geq 0$ such that every proper subset of $P$ is compatible. Equivalently, the full subcomplex of $\M(K)$ spanned by the vertices $(\sigma_0,\tau_0),\ldots,(\sigma_r,\tau_r)$ is the boundary $\partial\Delta^r$ of an $r$-simplex.

Note that an $f$-cycle has at least three primitive Morse functions. One with exactly three primitive Morse functions is said to be \emph{minimal} and two minimal $f$-cycles sharing exactly one regular pair are said to be \emph{adjacent}. From the mutually exclusive nature of properties \eqref{item:Property1DefinitionMorseFunction} and \eqref{item:Property2DefinitionMorseFunction} in page \pageref{item:Property1DefinitionMorseFunction} we see that no collection of regular pairs of a given combinatorial Morse function admits $f$-cycles of any index. Actually, Forman proved that this property characterizes the discrete vector fields that arise from a discrete Morse function (see \cite[Theorem 9.3]{For1}).

%The prefix ``$V$'' in the definition of $f$-path is a notational abuse since these paths are originally defined for some discrete vector field $V$ (and the regular faces in $P$ belong to $V$). Here we use it simply as a prefix not referencing any vector field.

\begin{obss}\label{Remarkssss}\mbox{}\begin{enumerate}
\item[($i$)] Note that a cycle $e_0,\ldots,e_r$ in the $1$-skeleton of a complex $K$ gives rise to two possible $f$-cycles of index $0$ in $K$: choosing a vertex $v_0$ for $e_0$, one of them is $\{(v_0,e_0),(v_1,e_1),\ldots,(v_r,e_r)\}$ where $v_i\neq v_{i+1}$ for all $i=0,\ldots,r-1$. The other $f$-cycle arises from selecting the other vertex of $e_0$ to be the source of the primitive Morse function.

\item[($ii$)] It is easy to see that if $\{(\sigma_1,\tau_1),(\sigma_2,\tau_2),(\sigma_3,\tau_3)\}$ is a minimal $f$-cycle of index $k-1$ then $\{\tau_1,\tau_2,\tau_3\}$ spans a complex with $k+2$ vertices and a complete $1$-skeleton.\end{enumerate}\end{obss}

The following result deals with the cases in which an isomorphism $\M(K)\to \M(L)$ does not restrict to an isomorphism $\M(K^{(1)})\to \M(L^{(1)})$.

\begin{prop}\label{Proposition:IsomorphismRestrictsToOneSkeleton} Let $K,L$ be  connected simplicial complexes and let $F:\M(K)\rightarrow\M(L)$ be a simplicial isomorphism. If there exists a primitive Morse function $(v,e)\in V_{\M(K)}$ of index $0$ such that $F(v,e)=(\sigma^{n-1},\tau^n)$ with $n\geq 2$, then $K=L=\partial\Delta^m$ for some $m\geq 2$.\end{prop}

\begin{proof} We may assume that $n$ is maximal with the property that there exists $(v,e)\in V_{\M(K)}$ of index $0$ whose image is $(\sigma^{n-1},\tau^n)$ for some $n\geq 2$. With this assumption, we shall prove that $K=\partial\Delta^{n+1}$.  Let $w$ be the other end of $e$ and consider $F(w,e)$. Since $n\geq 2$ it is not hard to build a primitive Morse function incompatible with $F(v,e)$ and $F(w,e)$ at the same time, thus $e$ must be a face of a $2$-simplex $\{v,w,u\} \in K$. Let $e'= wu$ and $e''=uv$ and consider the minimal $f$-cycle $\{(v,e),(w,e'),(u,e'')\}$ in $K$. Then $\{F(v,e),F(w,e'),F(u,e'')\}$ is a minimal $f$-cycle of index $n-1$ in $L$. Let $F(v,e)=(\sigma,\tau)$, $F(w,e')=(\sigma',\tau')$ and $F(u,e'')=(\sigma'',\tau'')$. A simple reasoning shows that if $\sigma'\prec\tau$ then the situation of Figure \ref{Figure:DemoProp2} would arise, which leads to a contradiction.
\begin{figure}[H]
\centering
\includegraphics[scale=0.7]{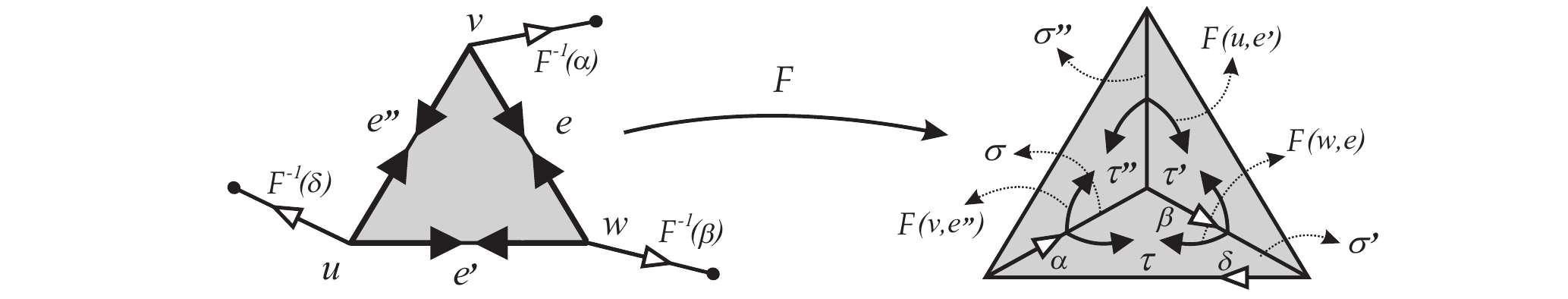}
\caption{The image of $(v,e'')$, $(u,e')$ and $(w,e)$ in the case $\sigma'\prec\tau$. If we consider a minimal $f$-cycle of index $n-1$ $\{\alpha,\beta,\delta\}$ in $\tau$ (in white arrows) then its preimage by $F$ does not constitute an $f$-cycle in $K$, which contradicts the fact that $F$ is an isomorphism.}
\label{Figure:DemoProp2}
\end{figure}
\noindent Therefore, we must have $\sigma\prec\tau''$ and the situation is as shown in Figure \ref{Figure:DemoProp1}. Let $Q$ be the subcomplex generated by the $n$-simplices $\tau,\tau',\tau''$ and note that $Q$ has $n+2\geq 4$ vertices and a complete $1$-skeleton (see Remark \ref{Remarkssss} ($ii$)). Let $S$ denote the collection of all primitive Morse function in $Q$ of index $0$ and let $G(x,a)=t(F^{-1}(x,a))\in K$ for each $(x,a)\in S$. We will prove that $K=\partial \Delta^{n+1}$ in various steps.

\begin{figure}[H]
\centering
\includegraphics[scale=0.7]{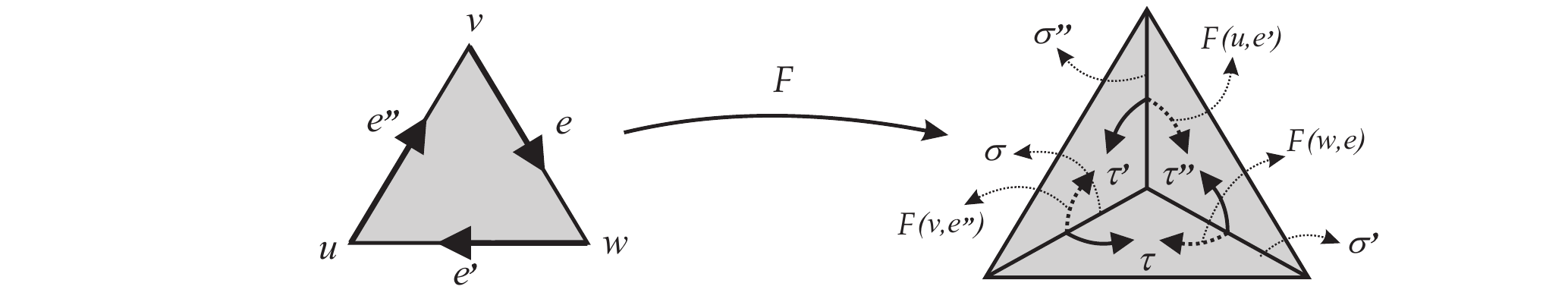}
\caption{}
\label{Figure:DemoProp1}
\end{figure}

\textsc{Step 1.} We show first that $G(S)$ is a collection of $k$-simplices for a fixed $k\leq n$. Consider a sequence $\tau=\eta^{n}\succ\sigma=\eta^{n-1}\succ\eta^{n-2}\succ\cdots\succ\eta^1\succ\eta^0=y$ of faces of the $n$-simplex $\tau$ ending in a vertex $y$ of $\tau$. Each pair $(\eta^{i-1},\eta^i)$ is incompatible with the previous and the next pair. Since incompatibility for a given regular pair only happens with regular pairs of one dimension up, one dimension down or of the same dimension, we conclude that $F^{-1}(y,\eta^1)=(\psi^{k-1},\rho^k)$ for some $k\leq n$. Now, since $Q$ has a complete 1-skeleton then any edge $a\in Q$ is part of a cycle also containing $\eta^1$. Therefore, any $(x,a)\in S$ is part of an $f$-cycle of index $0$ containing either $(y,\eta^1)$ or $(z,\eta^1)$, where $z$ is the other end of $\eta^1$ (see Remark \ref{Remarkssss} ($i$)). Since by definition $F$ maps $f$-cycles to $f$-cycles, it suffices to show that $t(F^{-1}(z,\eta^1))$ is also a $k$-simplex. But since $|V_Q|\geq 4$, we can form an $f$-cycle of index $0$ containing $(y,\eta^1)$ and a new pair $(p,\psi)$, and another one containing $(z,\eta^1)$ and $(p,\psi)$ as shown in Figure \ref{Figure:DemoFinal1}.

\begin{figure}[H]
\centering
\includegraphics[scale=0.7]{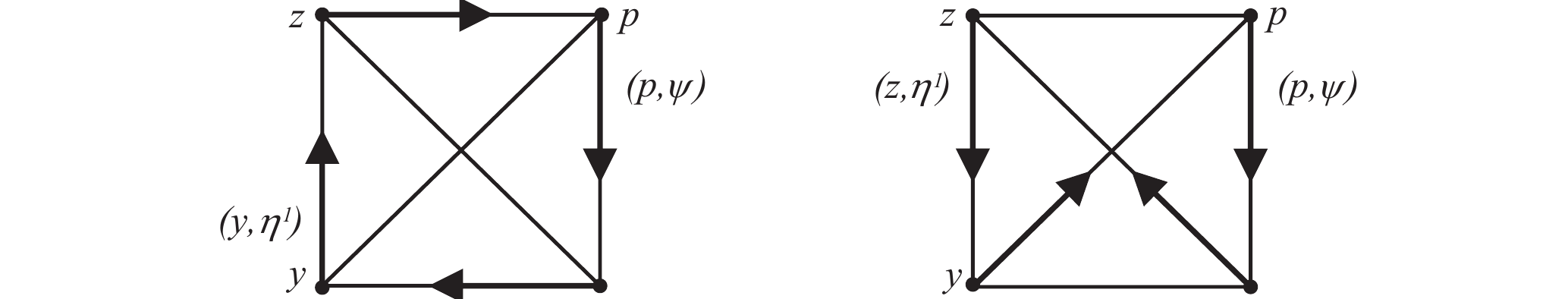}
\caption{}
\label{Figure:DemoFinal1}
\end{figure}

\textsc{Step 2.} We show that $k=n$ and that $G(S)$ spans $\partial\Delta^{n+1}$. Fix a minimal $f$-cycle $C_1=\{(v_1, v_1v_2),(v_2,v_2v_3),(v_3,v_1v_3)\}$ in $Q$ and let $T$ be the subcomplex of $K$ generated by the three $k$-simplices in $G(C_1)$. Note that $|V_{T}|=k+2$ by Remark \ref{Remarkssss} ($ii$). We claim that all $k$-simplices in $G(S)$ have their vertices in $V_{T}$. To see this, let $(x,a)\in S$ and let $y$ be the other end of $a$. All possible situations for $(x,a)$ with respect to $C_1$ are contemplated in Figure \ref{Figure:regularpairrespectC1} where one can verify that it is  always possible to find a sequence of adjacent minimal $f$-cycles between $C_1$ and a minimal $f$-cycle containing $(x,a)$.
\begin{figure}[h]
\centering
\includegraphics[scale=0.6]{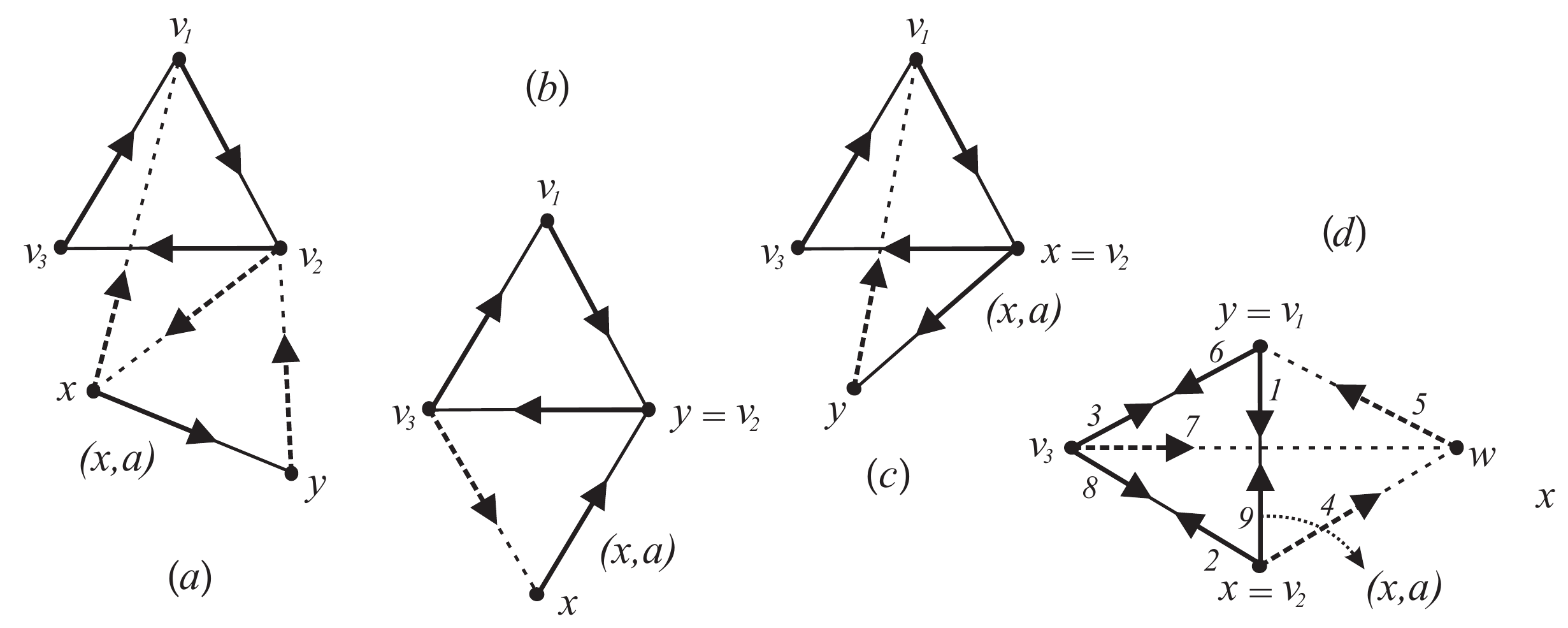}
\caption{The sequence of adjacent minimal $f$-cycles in situation ($d$) is given by $C_1=\{1,2,3\},\{1,4,5\},\{5,6,7\}$ and $\{6,8,9=(x,a)\}$.}
\label{Figure:regularpairrespectC1}
\end{figure}
By an inductive argument it suffices to show that the image by $G$ of a regular pair in a minimal $f$-cycle adjacent to $C_1$ has it vertices in $V_T$. Let $C_2=\{(v_2,v_2v_3),(v_3,v_3v_4),(v_4,v_2v_4)\}$ be a generic minimal $f$-cycle adjacent to $C_1$. Since the $k$-simplex $G(v_2,v_2v_3)\in G(C_1)\cap G(C_2)$, by Remark \ref{Remarkssss} ($ii$) it suffices to show  that the only vertex  $q\in V_{T}\setminus V_{G(v_2,v_2v_3)}$ is also in $G(C_2)$. But since $(v_3,v_3v_4)\nsim (v_3,v_1v_3)$ then either $s(F^{-1}(v_3,v_3v_4))=s(F^{-1}(v_3,v_1v_3))$ or $t(F^{-1}(v_3,v_3v_4))=t(F^{-1}(v_3,v_1v_3))$. The situation must be as shown in Figure \ref{Figure:DemoProp1} and the possible cases are shown in Figure \ref{Figure:demoprop43a}. This proves that $q\in G(C_2)$.

\begin{figure}[h]
\centering
\includegraphics[scale=0.8]{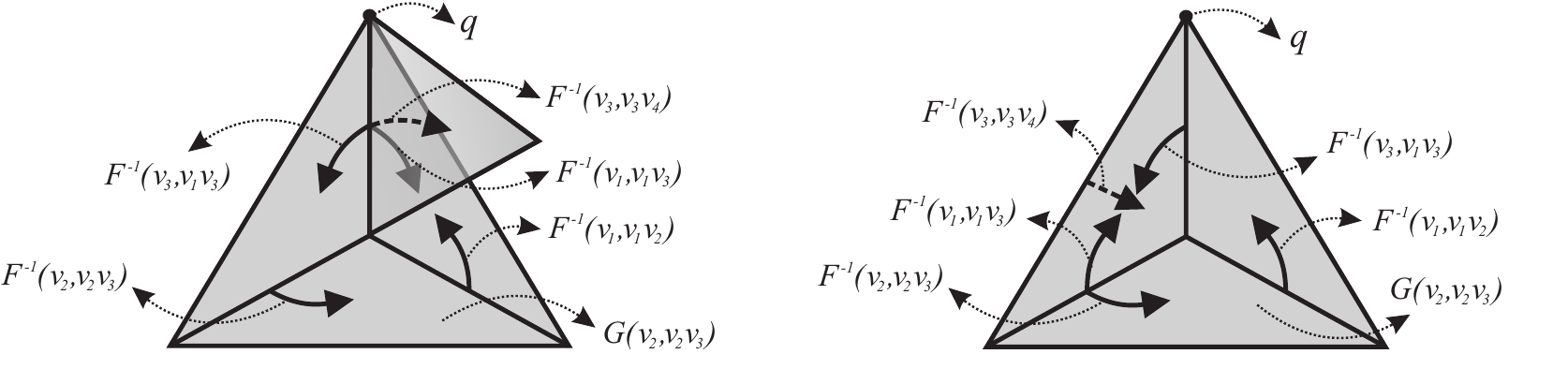}
\caption{Here $F^{-1}(v_3,v_3v_4)$ is drawn with a dashed arrow. On the left: the case $s(F^{-1}(v_3,v_3v_4))=s(F^{-1}(v_3,v_1v_3))$ cannot happen because we get more than $k+2$ vertices. On the right: in the case $t(F^{-1}(v_3,v_3v_4))=t(F^{-1}(v_3,v_1v_3))$ we readily see that $q\in G(C_2)$.}
\label{Figure:demoprop43a}
\end{figure}

Now, since $Q$ has a complete $1$-skeleton then we can form a cycle in $Q^{(1)}$ containing all the vertices of $Q$. The corresponding $f$-cycle of index $0$ has as a preimage by $F$ an $f$-cycle of index $k-1$ with $n+2$ regular pairs. By definition, the target of all these pairs are distinct $k$-simplices. Therefore, we conclude that $k=n$ and that $G(S)$ spans $\partial\Delta^{n+1}$.
 
\textsc{Step 3.} We show that $K$ is spanned by $G(S)$. First, note that two primitive Morse functions $(x,a),(x,b)\in S$ of index $0$ sharing the same source vertex $x\in V_Q$ are mapped by $F^{-1}$ to primitive Morse functions with the same target $n$-simplex (i.e. $G(x,a)=G(x,b)$). To see this, note that since  $F^{-1}(x,a)\nsim F^{-1}(x,b)$ then, by \textsc{Step 1}, either $s(F^{-1}(x,a))=s(F^{-1}(x,b))$ or $G(x,a)=G(x,b)$. Assume the first case holds and let $(x,c)\in S$ with $c\neq a,b$. Note that such a pair $(x,c)$ exists because $n\geq 2$. Since the only primitive Morse functions incompatible with both $F^{-1}(x,a)$ and $F^{-1}(x,b)$ have source $s(F^{-1}(x,a))=s(F^{-1}(x,b))$, there exists an $n$-simplex in $G(S)$ with one vertex $q$ not in  $V_{G(x,a)}\cup V_{G(x,b)}$ (see Figure \ref{Figure:demofinalpunto3}).
\begin{figure}[h]
\centering
\includegraphics[scale=0.7]{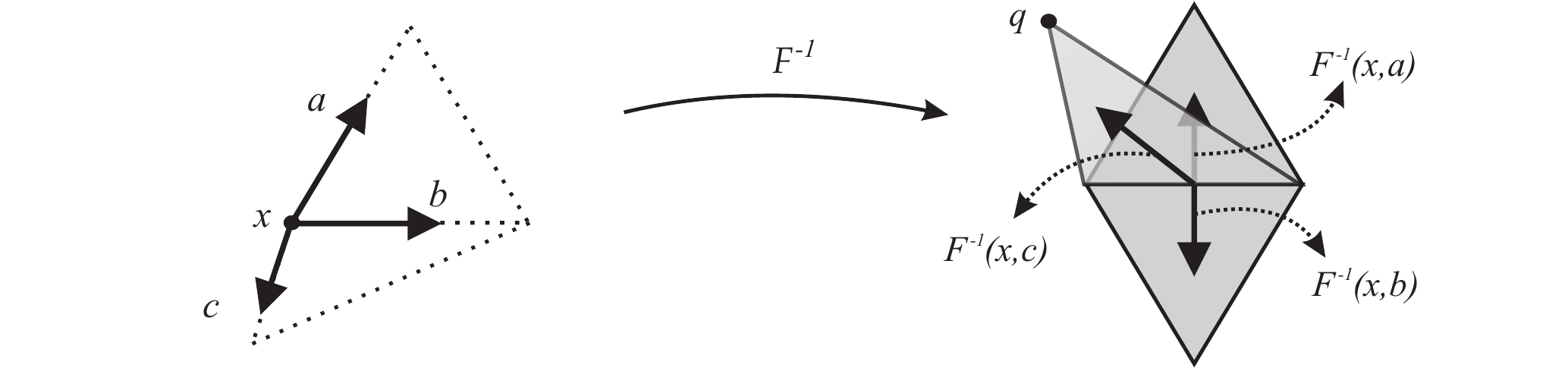}
\caption{}
\label{Figure:demofinalpunto3}
\end{figure}

This is a contradiction since, by the reasoning made in \textsc{Step 2}, all the vertices in $G(S)$ are included in the set of $n+2$ vertices determined by any two distinct $n$-simplices in $G(S)$. We conclude that $G(x,a)=G(x,b)$, thus we have a bijection between $V_Q$ and $G(S)$. Suppose now that $K-\langle G(S)\rangle\neq\emptyset$. Here $\langle G(S)\rangle$ denotes the subcomplex spanned by $G(S)$. Let $\tilde{e}\in K$ be an edge such that $\tilde{e}\cap\langle G(S) \rangle$ consists of a vertex $z$. Consider the primitive Morse function $(z,\tilde{e})$ and let $(z,\tilde{e}'),(z,\tilde{e}'')\in \langle G(S) \rangle$. Since $(z,\tilde{e})\nsim (z,\tilde{e}')$ and $(z,\tilde{e})\nsim (z,\tilde{e}'')$ then $F(z,\tilde{e})\nsim F(z,\tilde{e}')$ and $F(z,\tilde{e})\nsim F(z,\tilde{e}'')$. Since $n$ is maximal and $t(F(z,\tilde{e}'))=t(F(z,\tilde{e}''))$ by the previous reasoning, then $t(F(z,\tilde{e}))$ must be equal to $t(F(z,\tilde{e}'))=t(F(z,\tilde{e}''))$. 
This is a contradiction because, due to the bijection between $V_Q$ and $G(S)$, all $n+1$ regular pairs whose target is this $n$-simplex are in the image of the $n+1$ regular pairs in $S$ with source $z$. This concludes the proof.
\end{proof}

%\footnote{There are $n+2$ vertices in $Q$, so $deg_{Q^{(1)}}(z)=n+1$.}

Similarly as we did with the edges of simple graphs, for simplicity of notation, an $n$-simplex $\sigma=\set{v_0,\ldots,v_n}\in K$ will be denoted by $\sigma=v_0\cdots v_n$.
 
\begin{proof}[Proof of Theorem A] Let $F:\M(K)\rightarrow\M(L)$ be an isomorphism. By Proposition \ref{Proposition:IsomorphismRestrictsToOneSkeleton} we may assume that every primitive Morse function  of index $0$ in $\M(K)$ (resp. in $\M(L)$) is mapped by $F$ (resp. by $F^{-1}$) to a primitive Morse function of index $0$. This gives a well-defined isomorphism $F|_{\M(K^{(1)})}:\M(K^{(1)})\rightarrow\M(L^{(1)})$. By Theorem \ref{Theorem:MainForSimpleGraphs} there exists an isomorphism $f:K^{(1)}\to L^{(1)}$ with $f(v)=s(F(v,e))$ for any $e\succ v$. Note that for every edge $xy\in K$ we have $F(x,xy)=(f(x),f(x)f(y))$. We will show by induction that for any
$(n+1)$-simplex $v_0\cdots v_{n+1}$, 
 $$F(v_0\cdots v_{n},v_0\cdots v_{n} v_{n+1})=(f(v_0)\cdots f(v_{n}),f(v_0)\cdots f(v_{n})f(v_{n+1})).$$
  
Given $\tau=v_0\cdots v_{n+1}\in K$, consider the following two families of primitive Morse functions: \begin{itemize} 
\item $I=\{(v_0\cdots\hat{v}_i\cdots v_{n},v_0\cdots v_{n})\ , \ 0\leq i\leq n\}$ 
\item $J=\{(v_1\cdots\hat{v}_j\cdots v_{n+1},v_1\cdots v_{n+1})\ ,\ 1\leq j\leq n+1\}$,\end{itemize}
where the hat over a vertex means that that vertex is to be omitted. By induction, \small \begin{itemize}
\item $F(v_0\cdots \hat{v}_i\cdots v_{n},v_0\cdots v_{n})=(f(v_0)\cdots\widehat{f(v_i)}\cdots f(v_{n}),f(v_0)\cdots f(v_{n}))$ \normalsize and \small
\item $F(v_1\cdots\hat{v}_j\cdots v_{n+1},v_1\cdots v_{n+1})=(f(v_1)\cdots\widehat{f(v_j)}\cdots f(v_{n+1}),f(v_1)\cdots f(v_{n+1}))$. \end{itemize}
\normalsize
 Since $(v_0\cdots v_{n},v_0\cdots v_{n+1})\in\M(K)$ is incompatible with every element of $I$ then $$F(v_0\cdots v_{n},v_0\cdots v_{n+1})=(f(v_0)\cdots f(v_n),f(v_0)\cdots f(v_n)w)$$ for some vertex $w\in L$. On the other hand, since $(v_1\cdots v_{n+1},v_0 \cdots v_{n+1})\in\M(K)$ is incompatible with every element of $J$ then $$F(v_1\cdots v_{n+1},v_0 \cdots v_{n+1})=(f(v_1)\cdots f(v_{n+1}),f(v_1)\cdots f(v_{n+1})u)$$ for some vertex $u\in L$. But $(v_0\cdots v_{n},v_0\cdots v_{n+1})\nsim (v_1\cdots v_{n+1},v_0 \cdots v_{n+1})$, so we must have $f(v_0)\cdots f(v_n)w=f(v_1)\cdots f(v_{n+1})u$, and therefore $w=f(v_{n+1})$ and $u=f(v_0)$. 
\end{proof}

%\subsection*{Acknowledgement} The authors are indebted to Jonathan Barmak for many helpful discussions which lead to the present article.

%The authors are greatly indebted to Jonathan Barmak for many helpful discussions and contributions which lead to the present article.


\begin{thebibliography}{99}

\bibitem{AyFeQuVi1} R. Ayala, L. M. Fern\'andez, A. Quintero, and J. A. Vilches. \textit{A note on the pure Morse complex of a graph.} Topology Appl. 155 (2008), No. 17-18, 2084-2089.

%\bibitem{BaBjLiShWe} E. Babson, A. Bj\"orner, S. Linusson, J. Shareshian and V. Welker. \textit{Complexes of not $i$-connected graphs.} Topology 38 (1999), No. 2, 271-299

%\bibitem{BaHe} E. Babson and P. Hersh. \textit{Discrete Morse functions from lexicographic orders.} Trans. Amer. Math. Soc. 357 (2005), No. 2, 509-534 (electronic).

\bibitem{BaKo} E. Babson, D. N. Kozlov. \textit{Proof of the Lov\'asz conjecture.} Ann. of Math. (2) 165 (2007), No. 3, 965-1007.

\bibitem{BjWe} A. Bj\"orner, V. Welker. \textit{Complexes of directed graphs.} SIAM J. Discrete Math. 12 (1999), No. 4, 413-424 (electronic).
       
%\bibitem{Ca1} N. A. Capitelli. \textit{Colapsabilidad en variedades combinatorias y espacio de deformaciones.} Diploma Thesis, Universidad de Buenos Aires, 2009.

\bibitem{Cha} M. K. Chari. \textit{On discrete Morse functions and combinatorial decompositions}. Discrete Math. 217 (2000), No. 1-3, 101-113.

\bibitem{ChJo} M. K. Chari, M. Joswig. \textit{Complexes of discrete Morse functions.} Discrete Math. 302 (2005), No. 1-3, 39–51. 
           
\bibitem{Eng} A. Engstr\"om. \textit{Complexes of directed trees and independence complexes.} Discrete Math. 309 (2009), No. 10, 3299-3309.

\bibitem{For1} R. Forman. \textit{Morse theory for cell complexes.} Adv. Math. 134 (1998), No. 1, 90-145.

\bibitem{For2} R. Forman. \textit{Witten-Morse theory for cell complexes.} Topology 37 (1998), No. 5, 945-979.

%\bibitem{Hat} A. Hatcher. \textit{Algebraic Topology.} Cambridge University Press (2002)

\bibitem{Joj} D. Joji\'c.  \textit{Shellability of complexes of directed trees.} Filomat 27 (2013), No. 8, 1551-1559.
      
\bibitem{Koz1} D. N. Kozlov. \textit{Complexes of directed trees.} J. Combin. Theory Ser. A 88 (1999), No. 1, 112-122.

\bibitem{LuWe} A. Lundell, S. Weingram. \textit{The topology of CW complexes.} Van Nostrand - Reinhold, New York (1969).


%\bibitem{Koz2} D. N. Kozlov. \textit{Collapsibility of $\Delta(\Pi_n)/\mathcal{S}_n$ and some related CW complexes.} Proc. Amer. Math. Soc. 128 (2000), No. 8, 2253-2259.
             
%\bibitem{Min} E. G. Minian. \textit{Some remarks on Morse theory for posets, homological Morse theory and finite manifolds.} Topology Appl. 159 (2012), No. 12, 2860-2869.

%\bibitem{RoSa} C.P. Rourke, B.J. Sanderson. \textit{Introduction to piecewise-linear topology}. Springer-Verlag (1972).

%\bibitem{Sha} J. Shareshian. \textit{Discrete Morse theory for complexes of $2$-connected graphs.} Topology 40 (2001), No. 4, 681-701.

\end{thebibliography}
\end{document}